\documentclass[12pt]{amsart}
\usepackage{amsmath,amssymb,amsbsy,amsfonts,latexsym,amsopn,amstext,
                                               amsxtra,euscript,amscd}
\usepackage{url}
\usepackage[colorlinks,linkcolor=blue,anchorcolor=blue,citecolor=blue]{hyperref}
\usepackage{color}
\begin{document}

\newtheorem{theorem}{Theorem}
\newtheorem{lemma}[theorem]{Lemma}
\newtheorem{claim}[theorem]{Claim}
\newtheorem{cor}[theorem]{Corollary}
\newtheorem{prop}[theorem]{Proposition}
\newtheorem{definition}{Definition}
\newtheorem{question}[theorem]{Question}
\newcommand{\hh}{{{\mathrm h}}}

\def\sssum{\mathop{\sum\!\sum\!\sum}}
\def\ssum{\mathop{\sum\ldots \sum}}

\def \balpha{\boldsymbol\alpha}
\def \bbeta{\boldsymbol\beta}
\def \bgamma{{\boldsymbol\gamma}}
\def \bomega{\boldsymbol\omega}

\def\squareforqed{\hbox{\rlap{$\sqcap$}$\sqcup$}}
\def\qed{\ifmmode\squareforqed\else{\unskip\nobreak\hfil
\penalty50\hskip1em\null\nobreak\hfil\squareforqed
\parfillskip=0pt\finalhyphendemerits=0\endgraf}\fi}%%

%  use the AMS-Euler Fraktur fonts
%%%%%%%%%%%%%%%%%%%%%%%%%%%%%%%%%%
\newfont{\teneufm}{eufm10}
\newfont{\seveneufm}{eufm7}
\newfont{\fiveeufm}{eufm5}
%%%%%%%%%%%%%%%%%%%%%%%%%%%%%%%%%
%
%  allow automatic size selection in math mode
%
%%%%%%%%%%%%%%%%%%%%%%%%%%%%%%%%%
\newfam\eufmfam
     \textfont\eufmfam=\teneufm
\scriptfont\eufmfam=\seveneufm
     \scriptscriptfont\eufmfam=\fiveeufm
%%%%%%%%%%%%%%%%%%%%%%%%%%%%%%%%%
%
%  \frak works on a single symbol at a time...
%
\def\frak#1{{\fam\eufmfam\relax#1}}

\def\fK{\mathfrak K}
\def\fT{\mathfrak{T}}

\def\fA{{\mathfrak A}}
\def\fB{{\mathfrak B}}
\def\fC{{\mathfrak C}}
\def\fD{{\mathfrak D}}

\def\eqref#1{(\ref{#1})}

\def\vec#1{\mathbf{#1}}
\def\dist{\mathrm{dist}}
\def\vol#1{\mathrm{vol}\,{#1}}

\def\squareforqed{\hbox{\rlap{$\sqcap$}$\sqcup$}}
\def\qed{\ifmmode\squareforqed\else{\unskip\nobreak\hfil
\penalty50\hskip1em\null\nobreak\hfil\squareforqed
\parfillskip=0pt\finalhyphendemerits=0\endgraf}\fi}

%%%%%%%%%%%%%%%%%%%%%%%%%
% Alphabet calligraphie %
%%%%%%%%%%%%%%%%%%%%%%%%%
\def\cA{{\mathcal A}}
\def\cB{{\mathcal B}}
\def\cC{{\mathcal C}}
\def\cD{{\mathcal D}}
\def\cE{{\mathcal E}}
\def\cF{{\mathcal F}}
\def\cG{{\mathcal G}}
\def\cH{{\mathcal H}}
\def\cI{{\mathcal I}}
\def\cJ{{\mathcal J}}
\def\cK{{\mathcal K}}
\def\cL{{\mathcal L}}
\def\cM{{\mathcal M}}
\def\cN{{\mathcal N}}
\def\cO{{\mathcal O}}
\def\cP{{\mathcal P}}
\def\cQ{{\mathcal Q}}
\def\cR{{\mathcal R}}
\def\cS{{\mathcal S}}
\def\cT{{\mathcal T}}
\def\cU{{\mathcal U}}
\def\cV{{\mathcal V}}
\def\cW{{\mathcal W}}
\def\cX{{\mathcal X}}
\def\cY{{\mathcal Y}}
\def\cZ{{\mathcal Z}}
\newcommand{\rmod}[1]{\: \mbox{mod} \: #1}

\def\vr{\mathbf r}

\def\e{{\mathbf{\,e}}}
\def\ep{{\mathbf{\,e}}_p}
\def\em{{\mathbf{\,e}}_m}

\def\Tr{{\mathrm{Tr}}}
\def\Nm{{\mathrm{Nm}}}

 \def\SS{{\mathbf{S}}}

\def\lcm{{\mathrm{lcm}}}

\def\({\left(}
\def\){\right)}
\def\fl#1{\left\lfloor#1\right\rfloor}
\def\rf#1{\left\lceil#1\right\rceil}

\def\mand{\qquad \mbox{and} \qquad}

         \newcommand{\comm}[1]{\marginpar{\vskip-\baselineskip
%raise themarginpar a bit
         \raggedright\footnotesize
\itshape\hrule\smallskip#1\par\smallskip\hrule}}

%%%%%%%%%%%%%%%%%%%%%%%%%%%%%%%%%%%%%%%%%%%%%%%%%%%%%%%%
%%%%%%%%%%%%%%%%%%%%%%%%%%%%%%%%%%%%%%%%%%%%%%%%%%%%%%%%
%%%%%%%%%%%%%%%%%%%%%%%%%%%%%%%%%%%%%%%%%%%%%%%%%%%%%%%%
%%%%%%%%%%%%%%%%%%%%%%%%%%%%%%%%%%%%%%%%%%%%%%%%%%%%%%%%

%%%%%%%  END OF STANDARD STUFF %%%%%%%%%

%%%%%%%%%%%%%%%%%%%%%%%%%%%%%%%%%%%%%%%%%%%%%%%%%%%%%%%%
%%%%%%%%%%%%%%%%%%%%%%%%%%%%%%%%%%%%%%%%%%%%%%%%%%%%%%%%
%%%%%%%%%%%%%%%%%%%%%%%%%%%%%%%%%%%%%%%%%%%%%%%%%%%%%%%%
%%%%%%%%%%%%%%%%%%%%%%%%%%%%%%%%%%%%%%%%%%%%%%%%%%%%%%%
%%%%%%%%%%%
%%% Spell

\hyphenation{re-pub-lished}

\parskip 4pt plus 2pt minus 2pt

\mathsurround=1pt

\def\bfdefault{b}
\overfullrule=5pt

\def \F{{\mathbb F}}
\def \K{{\mathbb K}}
\def \Z{{\mathbb Z}}
\def \Q{{\mathbb Q}}
\def \R{{\mathbb R}}
\def \C{{\\mathbb C}}
\def\Fp{\F_p}
\def \fp{\Fp^*}

\title[Linear Congruences with Ratios]{Linear Congruences with Ratios}

 \author[I. E. Shparlinski] {Igor E. Shparlinski}
 \thanks{This work was  supported in part by ARC Grant~DP140100118}

\address{Department of Pure Mathematics, University of New South Wales,
Sydney, NSW 2052, Australia}
\email{igor.shparlinski@unsw.edu.au}

%% \date{ }

\begin{abstract} We use new bounds of double exponential sums 
with ratios of integers from prescribed intervals to get an asymptotic 
formula for the number of solutions to congruences 
$$
\sum_{j=1}^n a_j \frac{x_j}{y_j} \equiv a_0 \pmod p,
$$
with variables from rather general sets.   
\end{abstract}


\maketitle

 \section{Introduction}

 \subsection{Motivation} We count the number of solutions to a 
linear congruence  with rational variables  with restricted 
numerators and denominators. This includes solutions with 
rationals of a bounded height 
or more generally with a numerators and denominators from 
a certain large class of sets with a regular boundary. 
For example, this class of sets includes all convex sets.
In some special cases, the corresponding equation over $\Q$ has
recently been considered by Blomer and Br{\" u}dern~\cite{BlBr} 
and also by  Blomer, Br{\" u}dern 
and Salberger~\cite{BBS}. However, in positive characteristic
this natural question  has never been studied before. 
 
More precisley, for a prime $p$ we consider the equation
 \begin{equation}
\label{eq:Lin Eq a}
\sum_{j=1}^n a_j \frac{x_j}{y_j} =a_0,
\end{equation}
with  coefficients $\vec{a} = (a_0, a_1, \ldots, a_n) \in \F_p^{n+1}$ 
and   variables 
$$
\vec{x} = (x_1, \ldots, x_n), \ \vec{y} = (y_1, \ldots, y_n) \in \F_p^n,
$$ 
where $\F_p$ 
denotes the finite fields of $p$ elements. 

Given a set $\cS \subseteq [0, p-1]^{2n}$, we use 
$N(\vec{a}; \cS)$ 
to denote the number of solutions to the
equation~\eqref{eq:Lin Eq a} with variables 
$(x_1, y_1, \ldots, x_n, y_n)  \in \cS$.

%
%We also assume that we are given two cubic boxes  
%\begin{equation}
%\begin{split}
%\label{eq:boxes}
% \fA &= [A_1+1, A_1 + K]\times \ldots \times [A_n+1, A_n+K]\\
% \fB & =[B_1+1,B_1+L]\times \ldots \times [B_n+1, B_n+L],
%\end{split}
%\end{equation}
%with arbitrary integers  $A_j$, $B_j$, $i=1, \ldots, n$, and 
%positive integers  $K$ and $L$.  Our main 
%object of study is $N_n(\vec{a}; \fA,\fB)$ 
%which we defined as  the number of solutions to the
%equation~\eqref{eq:Lin Eq a} with variables 
%$$
%(x_1,\ldots,x_n) \in \fA \mand
%(y_1,\ldots,y_n) \in \fB.
%$$ 
%%Clearing the denominators, one can also write~\eqref{eq:Lin Eq a} in 
%%an equivalent form
%%as 
%%$$
%%\sum_{j=1}^n a_j  x_j \prod_{\substack{i=1\\i\ne j}}^n y_j =a_0 \prod_{i=1}^ny_j, 
%%\qquad y_j \ne 0, \ j =1, \ldots, n.
%%$$

The equation~\eqref{eq:Lin Eq a}
 can be considered over the integers. 
In particular,  Recently Blomer, Br{\" u}dern 
and Salberger~\cite{BBS} have studied it for 
$n=3$, $a_0 = 0$ and $a_1 = a_2 = a_3=1$.
In particular, by~\cite[Theorem~1]{BBS}, 
the number of integers solutions with $(\vec{x}, \vec{y}) \in [-H,H]^6$, 
to the  analogue of~\eqref{eq:Lin Eq a} with variables 
over $\Z$ is given by  $H^3Q(H) + O(H^{3-\delta})$, where  
 $Q \in \Q[X]$ is a polynomial of degree $4$ and $\delta > 0$ 
is some absolute constant.
Blomer and Br{\" u}dern~\cite{BlBr} have also suggested an
alternative approach which 
yields a tight upper bound for the same equation 
but for a slightly different way of ordering and counting solutions.
The methods  of~\cite{BlBr,BBS} can probably be extended 
to arbitrary $n$ (see, for example, the comment 
in~\cite[Section~1.3]{BBS}). 

In~\cite{Shp3}, a different approach 
has been suggested, which is based on some arguments from~\cite{Shp1}
and  leads
to bounds that are weaker by a logarithmic factor than those
expected to be produced by the methods of~\cite{BlBr,BBS},
however it seems to be more robust and is able to work in more general situations. 

Here we combine  some ideas from~\cite{Shp1} with several other arguments 
and apply them to  the case of the
equation~\eqref{eq:Lin Eq a} over a finite field.

Throughout the paper, any implied constants in the symbols $O$,
$\ll$  and $\gg$  may depend on 
%% the real  parameter $\varepsilon> 0$ and
the integer parameter $n\ge 1$. We recall that the
notations $U = O(V)$, $U \ll V$ and  $V \gg U$ are all equivalent to the
statement that the inequality $|U| \le c V$ holds with some
constant $c> 0$.

\subsection{Solutions in  boxes} 
We fix some intervals 
\begin{equation}
\label{eq:interv}
\cI_j = [A_j+1, A_j + K_j],\ \cJ_j = [B_j+1, B_j + L_j] \subseteq [0, p-1], 
\end{equation}
 with integers  $A_j$, $B_j$, $K_j$ and $L_j$, $j =1, \ldots, n$, 
and obtain the 
following  asymptotic formula. 

\begin{theorem}
\label{thm:NaW} 
For $n\ge 3$ and arbitrary intervals~\eqref{eq:interv}   
for the box $\cB = \cI_1 \times \cJ_1 \times\ldots \times \cI_1 \times \cJ_n$
 we have 
$$
\left|N(\vec{a}; \cB) - \frac{1}{p}  \prod_{j=1}^n(K_jL_j)\right|\le 
\sqrt{K_1  L_1 K_2 L_2} \prod_{j=3}^n  (K_j+ \sqrt{p L_j})  p^{o(1)} .
$$
\end{theorem}

We now consider the case when $\cB$ is a cube with the 
side length $H$.  

\begin{cor}
\label{cor:NaW H} 
For $n\ge 3$ and intervals~\eqref{eq:interv}  with 
$K_j = L_j = H$, 
$j =1, \ldots, n$, for the cubic box 
$\cC = \cI_1 \times \cJ_1 \times\ldots \times \cI_1 \times \cJ_n$
 we have 
$$
\left|N(\vec{a}; \cC) - \frac{H^{2n}}{p}\right| \le  p^{n/2-1+o(1)} H^{n/2+1}. 
$$ 
\end{cor}

In particular, the asymptotic formula of Corollary~\ref{cor:NaW H}
is nontrivial starting from the values of $H$ of order $p^{n/(3n-2) + \delta}$
for any fixed $\delta >0$ and sufficiently large $p$. 
%One can also derive  a similar result for $N(\vec{a}; \cD)$ 
%where $\cD = \cD_1 \times \ldots \times \cD_n$ with disks 
%$$
%\cD_j = \{(x,y)\in\Z^2:~ (x-A_j)^2 + (y-B_j)^2 \le H^2\}, \qquad 
%j =1, \ldots, n. 
%$$
We also record the following result which is convenient 
for further applications

For a set $\Omega \subseteq [0,1]^{2n}$ we use $p\Omega$ 
to denote its blow up   by $p$, 
that is, 
$$
p\Omega = \{p\bomega~:~ \bomega \in \Omega\}.
$$
Rounding up and down the sides of $p\Gamma$ for a cubic box 
\begin{equation}
\label{eq:box Pi}
\Gamma = [\alpha_1,\alpha_1 +\xi] \times [\beta_1,\beta_1 +\xi] \times \ldots \times [\alpha_n,\alpha_n +\xi] \times [\beta_n,\beta_n +\xi] \in [0,1]^{2n}, 
\end{equation}
we derive

\begin{cor}
\label{cor:NaW Gamma} 
For $n\ge 3$ and a  cubic box~\eqref{eq:box Pi} with $\xi> 1/p$ 
 we have 
$$
 \left|N(\vec{a}; p \Gamma )  - \xi^{2n} p^{2n-1}\right| \le  \(\xi^{2n-1} p^{2n-2} + \xi^{n/2+1}  p^{n} \)  p^{o(1)}.
$$
\end{cor}

\subsection{Solutions in well-shaped sets} 

We combine Corollary~\ref{cor:NaW H} with some ideas 
of Schmidt~\cite{Schm} to get an asymptotic formula for $N(\vec{a}; \Omega)$
for a rather general class of sets, which includes all convex sets. 

First we need to introduce some definitions. 
We define the  {\it distance\/} between a vector $\balpha \in [0,1]^m$
and a set $\varXi\subseteq [0,1]^m $  by
$$
\dist(\balpha,\varXi) = \inf_{\bbeta \in\varXi}
\|\balpha - \bbeta\|,
$$
where  $\|\bgamma\|$ denotes the {\it Euclidean norm\/} of $\bgamma$. Given
$\varepsilon >0$ and a set  $\varXi \subseteq [0,1]^m $ we define
the  sets
$$
\varXi_\varepsilon^{+} = \left\{ \balpha \in  [0,1]^m \backslash
\varXi \ : \ \dist(\balpha,\varXi) < \varepsilon \right\}$$
and
$$
\varXi_\varepsilon^{-} = \left\{ \balpha \in \varXi \ : \
\dist(\balpha,[0,1]^m \backslash \varXi )  < \varepsilon  \right\} .
$$

We note that in the definition of $\varXi_\varepsilon^{+}$ we discard the 
part of the outer $\varepsilon$-neighbourhood that does not belong to $[0,1]^m$.
These parts can also be included in $\varXi_\varepsilon^{+}$ but this does 
not affect our argument as we work only  with inner  $\varepsilon$-neighbourhoods 
$\varXi_\varepsilon^{-}$ and
$ \([0,1]^m \backslash \varXi\)_\varepsilon^{-}=\varXi_\varepsilon^{+}$.

Following~\cite{Shp3} (see also~\cite{Ker,KerShp}),  we say that a set $\varXi$ is {\it well-shaped\/} if 
for every $\varepsilon>0$ the  {\it Lebesgue measures\/} 
 $\mu\(\varXi_\varepsilon^{-}\)$ and  $\mu\(\varXi_\varepsilon^{+}\)$ exist, for some constant $C$,
and satisfy
\begin{equation}
\label{eq:Blowup}
\mu\(\varXi_\varepsilon^{\pm }\)  \le C \varepsilon.
\end{equation}
As we have mentioned, all convex sets are well-shaped. 

%For a set $\Omega \subsete [0,1]^m$ we use $p\Omega$ 
%to denote it blow up   by $p$, 
%that is, 
%$$
%p\Omega = \{p\bomega~:~ \bomega \in \Omega\}.
%$$

\begin{theorem}
\label{thm:NaOmega} 
For $n\ge 3$ and an arbitrary well-shaped set   $\Omega \subseteq [0,1]^{2n}$ 
of Lebesgue measure $\mu(\Omega)$,  we have 
$$
\left| N(\vec{a}; p \Omega) -  p^{2n-1} \mu(\Omega) \right|
 \le p^{2n-(5n-4)/(3n-2) +o(1)}.
%\le    
%\left\{\begin{array}{ll}
% p^{2n-2+o(1)},&\ \text{if }   n=3, 4\\
%p^{2n-5n/(3n-2) +o(1)},&\ \text{if }  n\ge 5. 
%\end{array}
%\right.
$$
\end{theorem}

\section{Preliminaries}

\subsection{Multiplicative congruences}

We recall the following special case of a result of
Ayyad, Cochrane and Zheng~\cite[Theorem~1]{ACZ} 
\begin{lemma}
\label{lem:4th Moment}  
Let $\cI_j, \cJ_j$, $j=1,2$, be four intervals as of the form~\eqref{eq:interv}
 $$
x_1y_2 \equiv x_2y_1 \pmod p, \quad   x_i\in \cI_i
\ y_i\in \cJ_i, \ i =1,2
$$
has $K_1K_2L_1L_2/p +  O\(\sqrt{K_1K_2L_1L_2} p^{o(1)}\)$ solutions. 
\end{lemma}

 We also need a version of the result of
Cilleruelo and Garaev~\cite[Theorem~1]{CillGar}.

\begin{lemma}
\label{lem:Concent Invers} For any  
integers $B$, $L$ and $M$ with $0\le B < B+L < p$ and $0 \le M < p$,   the congruence
$$
(B+y)z \equiv 1 \pmod p, \qquad B+1 \le y \le B+L, \ 1 \le z \le M
$$
has at most $p^{-1/2+o(1)} L^{1/2}M + p^{o(1)}$ solutions.
\end{lemma}

\begin{proof}  As in the proof of~\cite[Theorem~1]{CillGar} we note that 
by the Dirichlet principle, for any positive integers $U<p$ and $V$ with $UV \ge p$
one can choose integers $u$ and $v$ with 
$$
1\le u \le U, \qquad |v| = O(V), \qquad u B \equiv v \pmod p
$$
(see also~\cite[Lemma~3.2]{CSZ} for a more general statement). 
With this choice of $u$ and $v$ the above congruence
can be written as
$$
v z + u yz \equiv u \pmod p
$$
We now take $U = \rf{(p/L)^{1/2}}$ and $V = \rf{(pL)^{1/2}}$
(thus $UV \ge p$). 

Since the left hand side is at most $O(MV + LMU) = O((pL)^{1/2}M)$, we see that 
for every solution $(y,z)$ we have 
\begin{equation}
\label{eq:uz kp}
v z + u yz  = u + kp
\end{equation}
with some integer $k = O\((pL)^{1/2}M/p\) = O\(p^{-1/2} L^{1/2}M\)$.

We now recall the well-known bound
%\begin{equation}
%\label{eq:tau}
$$
\tau(m)  \le m^{o(1)},
$$
%%\end{equation}
on the number of integer positive divisors $\tau(m)$ 
of an integer $m\ne0$, 
see, for example,~\cite[Theorem~317]{HaWr}.
Since by~\eqref{eq:uz kp} we have the divisibility  $z\mid |u + kp|$
and also $0 < |u + kp| = O(p^2)$, we conclude that for each 
of the $O\(p^{-1/2} L^{1/2}M+1\)$ possible values of $k$,  there are 
at most $p^{o(1)}$ possible values for $z$, and thus for $y$. The 
result now follows.
\end{proof}

 \subsection{Exponential sums with ratios} 
For a prime $p$, we denote $\ep(z) = \exp(2 \pi i z/p)$.
Clearly for  $p \nmid u$ the expression $\ep(av/u)$
is correctly  defined (as $\ep(aw)$
for $w \equiv  v/u \pmod p$). 

 Let 
\begin{equation}
\label{eq:interv2}
\cI  = [A + 1, A + K],\ \cJ = [B + 1, B + L] \subseteq [0, p-1], 
\end{equation}
be  two intervals  with integers  $A$, $B$, $K$ and $L$. 

The following result is a variation of~\cite[Lemma~3]{Shp1}. We present 
it a slightly more general form that we need for our 
applications. 

\begin{lemma}
\label{lem:Sum Double}
 Let $\cI$ and $\cJ$ be 
two intervals of the form~\eqref{eq:interv2}  and
let $\cW\subseteq \cI\times\cJ$
be an arbitrary convex set.   Then uniformly over the 
integers $a$ with $\gcd(a,p)=1$, 
we have 
$$
 \sum_{(x,y) \in \cW} \ep(ax/y)  \ll  (K+ p^{1/2}L^{1/2})p^{o(1)}, 
$$
where the summation is over all integral points $(x,y) \in \cW$.
\end{lemma}

\begin{proof}
Since $\cW$ is convex,  for each $y$ we there  are
 integers $K \ge K_y > H_y \ge 1$ such that 
$$
 \sum_{(x,y) \in \cW} \ep(ax/y) =
\sum_{y \in \cJ} \sum_{x=A+H_y}^{A+K_y} \ep(ax/y).
$$
Following the proof of~\cite[Lemma~3]{Shp1}, we
define 
$$
I = \fl{\log(2p/K)} \mand J = \fl{\log (2p)}. 
$$

Furthermore, for a rational number $\alpha = u/v$ with $\gcd(v,p)=1$,
we denote by $\rho(\alpha)$ the unique integer $w$ with $w \equiv u/v \pmod p$
and $|w| < p/2$. 
Using the bound 
$$
\sum_{x=A+H_y}^{A+K_y} \ep(\alpha x) 
\ll  \min\left\{K, \frac{p}{|\rho(\alpha)|}\right\},
$$
which holds for any rational $\alpha$ with the denominator that is not a 
multiple of $p$
(see~\cite[Bound~(8.6)]{IwKow}), we obtain a version 
of~\cite[Equation~(1)]{Shp1}: 
\begin{equation}
\label{eq:W RT}
 \sum_{(x,y) \in \cW} \ep(ax/y)\ll K R  
+ p \sum_{j = I+1}^J T_{j} e^{-j},
\end{equation}
where
\begin{equation*}
\begin{split}
&R =\# \left\{y~:~B+1 \le y \le B+L,   \  
 |\rho(a/y)| < e^{I} \right\},\\
&T_{j} =\# \left\{y~:~B+1 \le y \le B+L, \  
e^j \le |\rho(a/y)| < e^{j+1} \right\}.
\end{split}
\end{equation*}

We now see that Lemma~\ref{lem:Concent Invers} implies the 
bounds
$$
  R \le  p^{-1/2+o(1)} L^{1/2}e^{I} + p^{o(1)} \le 
 p^{1/2+o(1)} L^{1/2}K^{-1}+ p^{o(1)}
$$
and
$$ 
  T_j \le p^{-1/2+o(1)} L^{1/2}e^{j} + p^{o(1)}.
$$
Substituting these bounds in~\eqref{eq:W RT}, we obtain 
\begin{equation*}
\begin{split}
&\left| \sum_{(x,y) \in \cW} \ep(ax/y)\right|\\
& \qquad  \ll  p^{1/2+o(1)} L^{1/2}+ K p^{o(1)}
+ p\sum_{j = I+1}^J\( p^{-1/2+o(1)} L^{1/2}e^{j} + p^{o(1)}\) e^{-j} \\
& \qquad    = p^{1/2+o(1)} L^{1/2}+ K p^{o(1)}
+   J p^{1/2+o(1)} L^{1/2}  + p^{1+o(1)}e^{-I}\\
& \qquad    = p^{1/2+o(1)} L^{1/2}+ K p^{o(1)},
\end{split}
\end{equation*}
which concludes the proof. 
\end{proof}

We also need a version of Lemma~\ref{lem:Sum Double}
on average over $a$. 

\begin{lemma}
\label{lem:Sum Double-Aver}
 Let $\cI$ and $\cJ$ be 
two intervals of the form~\eqref{eq:interv2}. Then, we have 
$$
\sum_{a=1}^{p-1} \left| \sum_{x \in \cI} \sum_{y \in \cJ} \ep(ax/y)\right|^2 
\le KLp^{1+o(1)}.
$$
\end{lemma}

\begin{proof}  First we write 
\begin{equation} 
\label{eq:Complete}
\sum_{a=1}^{p-1} \left| \sum_{x \in \cI} \sum_{y \in \cJ} \ep(ax/y)\right|^2 
=\sum_{a=0}^{p-1}\left|\sum_{x \in \cI} \sum_{y \in \cJ} \ep(ax/y)\right|^2 - K^2L^2.
\end{equation}

Expanding the square of the inner sum on the right hand side of~\eqref{eq:Complete}, changing the
order of summations and using the orthogonality of characters, we obtain 
$$
\sum_{a=0}^{p-1}\left|\sum_{x \in \cI} \sum_{y \in \cJ} \ep(ax/y)\right|^2 = \sum_{x_1,x_2 \in \cI} \sum_{y_1,y_2 \in \cJ} \sum_{a=0}^{p-1} \ep(a(x_1/y_1-x_2/y_2)) = p T, 
$$
where $T$ is the number of solutions to the congruence
\begin{equation} 
\label{eq:Cogr x/y}
x_1/y_1\equiv x_2/y_2 \pmod p, \qquad x_1,x_2 \in \cI, y_1,y_2 \in \cJ.
\end{equation}
Extending the admissible region of solutions to $\cI\times \cJ$ 
and evoking Lemma~\ref{lem:4th Moment}, we conclude that 
$$
T = \frac{K^2L^2}{p} +  O\(KL p^{o(1)}\)
$$
which together with~\eqref{eq:Complete} completes the 
proof. \end{proof}

\section{Proofs of Main Results}

\subsection{Proof of Theorem~\ref{thm:NaW}}

Using the orthogonality of the exponential function, we write 
$$
N(\vec{a}; \cB) 
  = \ssum_{(x_1, y_1, \ldots, x_n, y_n) \in \cB} \, \frac{1}{p}
\sum_{\lambda=0}^{p-1} \ep\(\lambda\(\sum_{j=1}^n a_j \frac{x_j}{y_j}- a_0 \)\).
$$
Changing the order of summation, and recalling the $\cB$ is a direct 
product of the intervals  $\cI_j$ and $\cJ_j$, $j =1, \ldots, n$, we obtain 
$$
N(\vec{a}; \cB)  =  \frac{1}{p} \sum_{\lambda=0}^{p-1} 
 \ep\(-\lambda a_0\) \prod_{j=1}^n \sum_{x_j \in \cI_j} \sum_{y_j \in \cJ_j} 
\ep\( \lambda a_j  x_j/y_j\).
$$
Now, the contribution from $\lambda=0$ gives the main term 
$$
 \frac{1}{p}  \prod_{j=1}^n \sum_{x_j \in \cI_j} \sum_{y_j \in \cJ_j} 1 = 
 \frac{1}{p}  \prod_{j=1}^n(K_jL_j) .
 $$
To estimate the error term, 
we apply Lemma~\ref{lem:Sum Double} to $n-2$ sums with $j = 3, \ldots, n$, 
getting
\begin{equation} 
\label{eq:N and W}
N(\vec{a}; \cB)  -  \frac{1}{p}  \prod_{j=1}^n(K_jL_j)\\
  \le  p^{-1+o(1)} 
\prod_{j=3}^n  (K_j+ p^{1/2}L_j^{1/2}) W,
\end{equation}
where 
$$
W = \sum_{\lambda=1}^{p-1} 
\left| \sum_{x_1 \in \cI_1} \sum_{y_1 \in \cJ_1}
\ep\( \lambda a_1  x_1/y_1\)\right|
\left| \sum_{x_2 \in \cI_2} \sum_{y_2 \in \cJ_2}
\ep\( \lambda a_2  x_2/y_2\)\right|. 
$$
Hence,  by the Cauchy inequality,  
\begin{equation} 
\label{eq:W1W2}
W \le \sqrt{W_1 W_2}, 
\end{equation}
where, for $\nu =1, 2$,  
$$
W_\nu = 
\sum_{\lambda=1}^{p-1} 
\left| \sum_{x_\nu \in \cI_\nu} \sum_{y_\nu \in \cJ_\nu} 
\ep\( \lambda a_\nu  x_\nu/y_\nu\)\right|^2
= \sum_{a=1}^{p-1} 
\left| \sum_{x_\nu \in \cI_\nu} \sum_{y_\nu \in \cJ_\nu} 
\ep\( a x_\nu/y_\nu\)\right|^2.
$$
We now apply Lemma~\ref{lem:Sum Double-Aver} to estimate $W_1$ 
and $W_2$ and see  from~\eqref{eq:W1W2} that 
$$
W \le \sqrt{K_1  L_1 K_2 L_2} p^{1+o(1)} 
$$
which together with~\eqref{eq:N and W}  concludes the proof. 

\subsection{Proof of Corollaries~\ref{cor:NaW H} and~\ref{cor:NaW Gamma}}
 
 For Corollary~\ref{cor:NaW H},  we see that  the first  terms, appearing 
in the bound of Theorem~\ref{thm:NaW} if $H^2$ 
while each term  
in the product  becomes
$O(p^{1/2} H^{1/2})$.  The result now follows. 

For Corollary~\ref{cor:NaW Gamma}, we approximate the set $p\Gamma$ 
by two cubes with side lengths $\fl{\xi p}$ and $\rf{\xi p}$.
Since $\xi > 1/p$, we have $(\xi p + O(1))^{2n} = (\xi p)^{2n}
  + O\((\xi p)^{2n-1}\)$. The result now follows from  
Corollary~\ref{cor:NaW H}.

\subsection{Proof of Theorem~\ref{thm:NaOmega}}
The proof follows the arguments of the proofs of~\cite[Theorem~1]{KerShp} 
or~\cite[Theorem~3.1]{Shp2} (however the concrete details are different).  

First we observe that since the complementary set $[0,1]^{2n} \setminus \Omega$ is also 
well-shaped, 
 it is enough to establish only the 
lower bound
\begin{equation}
\label{eq:LB}
N(\vec{a}; p \Omega) \ge  \frac{N(p \Omega)}{p}
 + O\(p^{2n-4/3+o(1)}\).
\end{equation}

We now recall some constructions and arguments from the proof of~\cite[Theorem~2]{Schm}. 
Pick a point $\boldsymbol\alpha  = (\alpha_1, \ldots, \alpha_{2n}) \in [0,1]^{2n}$ such that all its 
coordinates are irrational. 
For a positive integer $k$, let $\fC(k)$ be the set of cubes of the form
$$
\left[\alpha_1 + \frac{u_1}{k}, \alpha_1 + \frac{u_1+1}{k}\right]\times \ldots  \times \left[\alpha_{2n} + \frac{u_{2n}}{k}, \alpha_{2n} + \frac{u_{2n}+1}{k}\right],$$
with $ u_1, \ldots, u_{2n}\in \Z$.

We consider the set of points
\begin{equation}
\label{eq:Points}
\(\frac{x_1}{p}, \frac{y_1}{p}, \ldots, \frac{x_n}{p},\frac{y_n}{p}\) \in [0,1]^{2n}
\end{equation}
taken over all solutions $(\vec{x}, \vec{y}) \in \F_p^{2n}$  to 
the equation~\eqref{eq:Lin Eq a}

Note that the above irrationality condition on $\boldsymbol\alpha$ guarantees that the 
points~\eqref{eq:Points} 
all belong to the interior of the cubes from $\fC(k)$. 

Furthermore, let $\fC_0(k)$ be the set of cubes from $\fC(k)$ that are contained inside
of $\Omega$. By~\cite[Equation~(9)]{Schm},  for
any well-shaped set $\Omega\in [0,1]^{2n}$,  we have
\begin{equation}
\label{eq:Card Ck}
\#\fC_0(k) = k^{2n}\mu(\Omega) + O(k^{2n-1}).
\end{equation}
Let $\fB_1 = \fC_0(2)$ and for $i =2,3, \ldots$, 
let $\fB_i$ be the set of cubes $\Gamma \in \fC_0(2^i)$ that are not contained in any cube from $\fC_0(2^{i-1})$.
Clearly
\begin{equation}
\label{eq:Ci C0}
2^{-2in} \# \fB_i  + 2^{-2(i-1)n} \#\fC_0(2^{i-1})  \le \mu(\Omega), \qquad i=2,3, \ldots.
\end{equation}
We now infer from~\eqref{eq:Card Ck} that 
\begin{equation*}
\begin{split}
 \mu(\Omega) &- 2^{-2(i-1)n} \#\fC_0(2^{i-1}) \\
 & =
  \mu(\Omega) - 2^{-2(i-1)n} \(2^{2(i-1)n}\mu(\Omega) + O(2^{(i-1)(2n-1)})\)\\
 & \ll 2^{(i-1)(2n-1)- 2(i-1)n}= 2^{-i+1}. 
\end{split}
\end{equation*}
Therefore, the inequality~\eqref{eq:Ci C0} implies the bound
\begin{equation}
\label{eq:Card}
\# \fB_i \ll 2^{i(2n-1)}.
\end{equation}
We also see that for any integer $M\ge 1$,
\begin{equation}
\label{eq:Approx}
\Omega  \setminus \Omega_\varepsilon^{-}  \subseteq \bigcup_{i=1}^{M}  \bigcup_{\Gamma \in \fB_i}\Gamma  \subseteq \Omega
\end{equation}
with $\varepsilon =  (2n)^{1/2}2^{-M}$.  Indeed, for any  point 
$\bgamma \in \Omega  \setminus \Omega_\varepsilon^{-}$ 
there is a cube $\Gamma_\bgamma  \in \fC(2^{M})$ with 
$\bgamma  \in \Gamma$ (since for any  integer $k\ge 1$, the cubes from 
$\fC(k)$ tile the whole space $\R^{2n}$). Because the diameter (that is, the largest 
distance between the points) of $\Gamma_\bgamma $  is  $(2n)^{1/2}2^{-M}$, we see from the definition 
of $\Omega_\varepsilon^{-}$ that $\Gamma_\bgamma  \cap [0,1]^{2n} \backslash \Omega = \emptyset$.
Thus  $\Gamma_\bgamma  \subseteq \Omega$. This implies 
$$
\Gamma_\bgamma   \subseteq \bigcup_{i=1}^{2n}  \bigcup_{\Gamma \in \fB_i}\Gamma 
$$
and~\eqref{eq:Approx} follows.

Since $\Omega$ is well-shaped, from~\eqref{eq:Blowup} we deduce that
\begin{equation}
\label{eq:Aprox}
 \mu\(\bigcup_{i=1}^{2n}  \bigcup_{\Gamma \in \fB_i}\Gamma\) =  \mu\(\Omega\) + O(2^{-M}).
\end{equation}

We now assume that 
\begin{equation}
\label{eq:Small M}
2^{M} < p
\end{equation}
so Corollary~\ref{cor:NaW Gamma} applies to all cubes $\Gamma \in \fC_0(2^i)$, 
$ i =1, \ldots, M$. 
Together with~\eqref{eq:Approx}, this implies the inequality:
\begin{equation}
\label{eq:Prelim 1}
N(\vec{a}; p \Omega) \ge \sum_{i=1}^M \sum_{\Gamma \in \fB_i}N(\vec{a}; p \Gamma) 
=  p^{2n-1} \sum_{i=1}^M \sum_{\Gamma \in \fB_i} \mu(\Gamma) + O(R),
\end{equation}
where
$$R
=   \sum_{i=1}^M  \# \fB_i
\( 2^{-i(2n-1)} p^{2n-2} + 2^{-i(n/2+1)}  p^{n}\) p^{o(1)}.  
$$
We see from~\eqref{eq:Aprox}  that 
\begin{equation}
\begin{split}
\label{eq:Main T}
 p^{2n-1}\sum_{i=1}^M \sum_{\Gamma \in \fB_i} \mu(\Gamma)& = p^{2n-1}  \mu\(\bigcup_{i=1}^M  \bigcup_{\Gamma \in \fB_i}\Gamma\) \\
 & = 
 p^{2n-1}\mu\(\Omega\) +  O\(p^{2n-1}2^{-M}\).
 \end{split} 
\end{equation}
 
Furthermore, using~\eqref{eq:Card}, we derive
\begin{equation}
\label{eq:Error T1}
\begin{split}
R &\le    \sum_{i=1}^M  \(  p^{2n-2} + 2^{i(3n/2-2)}  p^{n}   \) p^{o(1)}\\
& = \(M p^{2n-2} + 2^{M(3n/2-2)}  p^{n}   \) p^{o(1)}.
\end{split} 
\end{equation}
Substituting~\eqref{eq:Main T} and~\eqref{eq:Error T1} 
in~\eqref{eq:Prelim 1} with the above  choice of $M$, noticing that~\eqref{eq:Small M}
implies $M = O(\log p)$, we obtain 
\begin{equation}
\label{eq:Prelim 2}
N(\vec{a}; p \Omega)\ge p^{2n-1}\mu\(\Omega\)  - Q p^{o(1)},
\end{equation}
where 
\begin{equation}
\label{eq:Error T2}
Q \le p^{2n-1}2^{-M}+ p^{2n-2} + 2^{M(3n/2-2)}  p^{n}  .
\end{equation}
We now choose $M$ to satisfy 
$$
2^{M}\le  p^{2(n-1)/(3n-2)}  < 2^{M+1}, 
$$
which asymptotically optimises the 
right hand side of the bound~\eqref{eq:Error T2}, 
verifies~\eqref{eq:Small M} and 
produces  to the bound $Q\ll p^{2n-(5n-4)/(3n-2)}+p^{2n-2} \ll p^{2n-5n/(3n-2)}$.
We now see from~\eqref{eq:Prelim 2} that~\eqref{eq:LB} holds, which concludes the proof. 

\section{Comments}

We note that for $B_1 = \ldots = B_n =0$, using~\cite[Lemma~3]{Shp1}
instead of Lemma~\ref{lem:Sum Double} in this special case one can 
improve  Theorem~\ref{thm:NaW} as follows
\begin{equation*}
\begin{split}
&\left| N(\vec{a}; \cB) - \frac{1}{p}  \prod_{j=1}^n(K_jL_j)\right|
\\
&\qquad \quad  \le 
\(\frac{K_1 L_1}{p^{1/2}} + \sqrt{K_1  L_1} \)  
\(\frac{K_2 L_2}{p^{1/2}} + \sqrt{K_2 L_2}\) \prod_{j=3}^n  (K_j+  L_j)  p^{o(1)}. 
\end{split}
\end{equation*}

Furthermore, it is easy to see that one can get a version of  
 Lemma~\ref{lem:Sum Double-Aver} for the more general sums of 
 Lemma~\ref{lem:Sum Double}, which becomes 
$$
\sum_{a=1}^{p-1} \left| \sum_{(x,y) \in \cW} \ep(ax/y)\right|^2 
\le K^2L^2 + KLp^{1+o(1)}, 
$$
that is, there is no cancellation of the main term 
for the number of solutions to the congruence~\eqref{eq:Cogr x/y} anymore. 
Thus the same arguments lead to the following result.
For $n\ge 3$ and arbitrary intervals~\eqref{eq:interv}   
and arbitrary convex sets $\cW_j \subseteq \cI_j\times \cJ_j$, 
$j =1, \ldots, n$, for the set $\cS = \cW_1 \times \ldots \times \cW_n$
 we have 
\begin{equation*}
\begin{split}
&\left|N(\vec{a}; \cS) - \frac{N(\cS)}{p}\right|\\
&  \qquad 
\le 
\(\frac{K_1 L_1}{p^{1/2}} + \sqrt{K_1  L_1} \)  
\(\frac{K_2 L_2}{p^{1/2}} + \sqrt{K_2 L_2}\) \prod_{j=3}^n  (K_j+ \sqrt{p L_j})  p^{o(1)} ,
\end{split}
\end{equation*}
where $N(\cS) = \#\(\cS \cap \Z^{2n}\)$. For example, this can be 
used for counting solutions to the
equation~\eqref{eq:Lin Eq a} with variables in disks
$$
(x_j - b_j)^2 + (y_j - c_j)^2 \le r_j^2, 
\qquad j =1, \ldots, n.
$$ 

One can also ask about solutions to~\eqref{eq:Lin Eq a} 
with additional co-primlaity condition $\gcd(x_j,y_j)=1$, 
$j =1, \ldots, n$, 
that is, essentially in {\it Farey fractions\/}. Using simple 
inclusion-exclusion arguments, one can easily derive relevant 
asymptotic formulas from our results.

Finally, we remark that Lemma~\ref{lem:Sum Double} can be viewed as a 
statement about cancellations among short Kloosterman sums 
of the form
$$
\cK(\lambda;\cJ) = \sum_{u \in \cJ} \ep(\lambda/u)
$$
over an interval $\cJ = [B + 1, B + L]$ 
when $\lambda$ runs over an interval $\cI  = [A + 1, A + K]$. 
Say, for $K = L$ we have a nontrivial cancellation starting 
with $L \ge p^{1/3 + \delta}$ for any fixed $\delta >0$,
 which is beyond the range of 
modern bounds of individual sums short Kloosterman sums
over intervals that are not at the origin, 
we refer to  the recent work of
Bourgain and Garaev~\cite{BouGar} for an outline of the state
of art and  several results.

%
%For $n=3$, the result of Blomer, Br{\" u}dern and Salberger~\cite{BBS}
%gives an asymptotic formula for the number of solutions to~\eqref{eq:Lin Eq a}
%with $a_1=a_2 =a_3=1$ and $a_0=0$. It is also noted in~\cite{BBS} that the same 
%method works for any $n$ and also for the relatively prime variables
%$\gcd(y_j,x_j)=1$. However it is not clear whether this method works 
%for   equations with $a_0 \ne 0$. 

%\hrule 
%
%Tim's bounds are better for $n \le 4$?
%
%Any nontrivial lower bounds?

%
%Instead of the $\tau$, one can work with Hooley's $\Delta$-function 
%as we look for divisors $z \in [e^j, e^{j+1}$.
%But the we need a bound on 
%$$
%\sum_{m \le Q} \Delta(m)^n,
%$$
%which is better than that for $\tau$. 
%% 
%\section*{Acknowledgement}
%
%The author is grateful to Valentin Blomer and Tim Browning for very 
%useful discussions. 
%This work was  supported in part by ARC Grant~DP140100118.

\end{document}